\documentclass[a4paper,12pt]{amsart}
\usepackage[mathscr]{eucal}
\usepackage{amssymb}
\usepackage{amsthm}
\theoremstyle{plain}
\newtheorem{theorem}{Theorem}[section]
\newtheorem{conjecture}{Conjecture}
\newtheorem{corollary}{Corollary}[section]
\newtheorem{lemma}{Lemma}[section]
\newtheorem{proposition}{Proposition}[section]

\theoremstyle{remark}
\newtheorem{remark}{Remark}[section]

\DeclareMathOperator{\Ric}{Ric}
\DeclareMathOperator{\Rm}{Rm}
\DeclareMathOperator{\tr}{tr}
\DeclareMathOperator{\Vol}{Vol}
\DeclareMathOperator{\codim}{codim}
\DeclareMathOperator{\rank}{rank}
\DeclareMathOperator{\Id}{Id}
\DeclareMathOperator{\Exp}{Exp}
\DeclareMathOperator{\secop}{sec}

\newcommand{\R}{\mathbb{R}}
\newcommand{\Sp}{\mathbb{S}^{p}}
\newcommand{\calH}{\mathcal{H}}
\newcommand{\calN}{\mathcal{N}}
\newcommand{\calV}{\mathcal{V}}
\newcommand{\calM}{\mathcal{M}}
\newcommand{\calS}{\mathcal{S}}
\newcommand{\eH}{\mathrm{H}}
\newcommand{\eN}{\mathrm{N}}
\newcommand{\eZ}{\mathrm{Z}}
\newcommand{\eB}{\mathrm{B}}
\newcommand{\II}{\mathrm{II}}
\newcommand{\dd}{\,d}
\newcommand{\restr}[2]{\left.#1\right|_{#2}}

\makeatletter
\@addtoreset{equation}{section}

\makeatother
\usepackage[hidelinks]{hyperref}

\title[Totally Geodesic Focal Sets]{Rigidity of Constant-Scalar-Curvature Ricci Shrinkers
	with Totally Geodesic Focal Set}
\author[F. Li]{Fengjiang Li}
\address[Fengjiang Li]
{Mathematical Science Research Center, Chongqing University of Technology, Chongqing 400054, China}
\email{fengjiangli@cqut.edu.cn}
\author[W. Li]{Wenqi Li}
\address{School of Mathematical Sciences, Key Laboratory of MEA (Ministry of
	Education) and   Shanghai Key Laboratory of PMMP, East China Normal University, Shanghai
	200241, China}
\email{51255500054@stu.ecnu.edu.cn}

\author[J. Ou]{Jianyu Ou}
\address[Jianyu Ou]
{Department of Mathematics, Xiamen University, Xiamen 361005, China}
\email{oujianyu@xmu.edu.cn}

\author[Y. Qu]{Yuanyuan Qu}
\address[Yuanyuan Qu]{School of Mathematical Sciences,Shanghai Key Laboratory of PMMP, East China Normal University, Shanghai 200241,
	China}
\email{52285500012@stu.ecnu.edu.cn}
\author[G. Wu]{Guoqiang Wu}
\address{School of Science, Zhejiang Sci-Tech University, Hangzhou 310018, China}
\email{gqwu@zstu.edu.cn}

\author[D. Zhou]{Detang Zhou}
\address{Instituto de Matem\'atica e Estat\'istica, Universidade Federal Fluminense, S\~ao Domingos, Niter\'oi, RJ 24210-201, Brazil}
\email{zhoud@id.uff.br}


\begin{document}
\raggedbottom

\begin{abstract}
Let $(M^n,g,f)$ be a complete connected shrinking gradient Ricci soliton
with constant scalar curvature, satisfying
\[
 \Ric+\nabla^2f=\frac12g.
\]
We prove that if the focal set of $f$ is totally geodesic, then $M$ is rigid. Equivalently, its
universal cover is a product of a compact Einstein manifold and a
Euclidean space. 
\end{abstract}
\maketitle

\section{Introduction}\label{sec:introduction}

A shrinking gradient Ricci soliton is a triple $(M^n,g,f)$ satisfying,
after normalization,
\begin{equation}\label{eq:soliton}
 \Ric+\nabla^2f=\frac12g.
\end{equation}
Here $\Ric$ and $\nabla^2f$ denote the Ricci tensor and Hessian of the
potential function $f$, and $R$ denotes scalar curvature.

Ricci flow was introduced by Hamilton~\cite{Hamilton}. Gradient Ricci
solitons generate self-similar solutions by scaling and diffeomorphisms,
and shrinking solitons play a central role in the study of singularities
and Perelman's entropy~\cite{Perelman}. Their classification is therefore
a basic problem in geometric analysis. We refer to Cao~\cite{CaoSurvey}
and Eminenti--La Nave--Mantegazza~\cite{EminentiLaNaveMantegazza} for
background from the geometric and elliptic viewpoints.

Classification results in low dimensions provide a starting point for
this problem. In particular, Ni--Wallach~\cite{NiWallach} developed an
elliptic approach to three-dimensional shrinking solitons and obtained
higher-dimensional results under vanishing Weyl curvature. In general
dimensions, additional geometric conditions have led to several
rigidity theorems. The constant-scalar-curvature problem asks how much
of this rigidity is already encoded by the scalar curvature alone.

A natural class of examples is the product of a positive Einstein
manifold and a Gaussian shrinking soliton:
\[
 (N^k\times\R^m,g_N+g_{\R^m},f),\qquad
 \Ric_N=\frac12g_N,\qquad f(x,y)=\frac{|y|^2}{4}.
\]
Following Petersen--Wylie~\cite{PetersenWylie}, a shrinking soliton is
called \emph{rigid} if it is a quotient of such a product, allowing either
factor to have dimension zero. Every rigid shrinker has constant scalar
curvature. The converse is the following conjecture of Cao; see,
for example,~\cite{ChengZhou}.

\begin{conjecture}[Cao's conjecture ]\label{conj:cao}
A complete shrinking gradient Ricci soliton has constant scalar
curvature if and only if it is rigid.
\end{conjecture}

Several results support this conjecture. Petersen--Wylie~\cite{PetersenWylie}
proved that constant scalar curvature together with radial flatness,
$\Rm(\cdot,\nabla f)\nabla f=0$, characterizes rigidity.
Fern\'andez-L\'opez--Garc\'ia-R\'io~\cite{FernandezLopezGarciaRio}
showed that, under the constant-scalar-curvature assumption, constant
rank of the Ricci tensor is sufficient. They also proved that the possible
scalar curvatures are integer multiples of the soliton constant.

A related direction imposes differential conditions on curvature.
Fern\'andez-L\'opez--Garc\'ia-R\'io~\cite{FernandezLopezGarciaRioHarmonic}
and Munteanu--\v{S}e\v{s}um~\cite{MunteanuSesum} established rigidity
results for shrinkers with harmonic Weyl tensor. Cao--Chen~\cite{CaoChen}
classified complete Bach-flat shrinkers in dimensions at least four.
Catino--Mastrolia--Monticelli~\cite{CatinoMastroliaMonticelli} subsequently
proved rigidity under a fourth-order divergence-free condition on the
Weyl tensor. These results exhibit several curvature conditions that
force an Einstein--Euclidean splitting.

In dimension four, Cheng--Zhou~\cite{ChengZhou} established the remaining
case of the conjecture, completing the constant-scalar-curvature
classification. Ou--Qu--Wu~\cite{OuQuWu} later gave another proof of
the four-dimensional $R=1$ case using the maximum principle.
In dimension five, Li--Ou--Qu--Wu~\cite{LiOuQuWu} proved rigidity when
$R=3/2$, identifying the shrinker as a finite quotient of
$\R^2\times\mathbb S^3$.

In higher dimensions, Li--Qu--Wu~\cite{LiQuWu} obtained
rigidity under a radial Ricci inequality outside a compact set and smooth
convergence to $\R^2\times\mathbb S^{n-2}$. These results motivate the
search for geometric conditions that imply rigidity in arbitrary
codimension.

In this paper we prove Cao's conjecture under the assumption that the
focal set is totally geodesic, in arbitrary dimension and for every
possible constant scalar curvature. Our argument applies to the general
range $R=k/2$ and $n=k+m$, with $k,m\ge2$; the remaining cases follow
from known rigidity results recalled in Section~\ref{sec:preliminaries}.

To explain the geometric condition, normalize the potential by
$|\nabla f|^2=f$ and let $M_-=f^{-1}(0)$ be its minimum set. For a rigid
product, $M_-=N\times\{0\}$ is totally geodesic. Our main result shows
that this property is sufficient for rigidity. We first record the
global focal structure used in the proof.

\begin{theorem}[Automatic global focal structure]\label{thm:automatic-focal}
	Let $(M^{k+m},g,f)$ be complete and connected, with $k,m\ge 2$, and suppose
	\[
	\Ric+\nabla^2 f=\frac12g,
	\qquad
	R\equiv\frac k2.
	\]
	After translating $f$ so that $|\nabla f|^2=f$, the set $M_-=f^{-1}(0)$ is a compact connected smooth submanifold of dimension $k$. Moreover
	\begin{equation}
		\begin{gathered}
 r:=2\sqrt f=d(M_-,\cdot),\qquad
 \restr{\Ric^\sharp}{M_-}=\frac12\Pi_{TM_-},\\
 \Exp^\perp:\nu M_-\longrightarrow M\quad\text{is a diffeomorphism}.
 \end{gathered}\label{eq:focal-package}
	\end{equation}
\end{theorem}

With this focal structure in place, we can state the main rigidity theorem.

\begin{theorem}[Arbitrary-codimension focal rigidity]\label{thm:focal-rigidity}
	Let $(M^{k+m},g,f)$ be a complete connected shrinking gradient Ricci soliton satisfying
	\[
	\Ric+\nabla^2 f=\frac12g,
	\qquad
	R\equiv\frac k2,
	\]
	with $k,m\ge 2$. Normalize $f$ by $|\nabla f|^2=f$ and put $M_-=f^{-1}(0)$. If $M_-$ is totally geodesic, then
	\begin{equation}
		(\widetilde M,\widetilde g)\cong(N^k,g_N)\times\R^m,
		\qquad
		\Ric_{g_N}=\frac12g_N,\label{eq:global-splitting}
	\end{equation}
	where $N$ is compact and simply connected.
\end{theorem}

Combining Theorem~\ref{thm:focal-rigidity} with the known boundary
cases gives the following form of Cao's conjecture.

\begin{corollary}\label{cor:cao-general}
Let $(M^n,g,f)$ be a complete connected shrinking gradient Ricci soliton
with constant scalar curvature, satisfying \eqref{eq:soliton}.
Normalize $f$ by $|\nabla f|^2=f$. If $M_-=f^{-1}(0)$ is totally geodesic,
then $M$ is rigid.
\end{corollary}

The corollary imposes no restriction on the dimension or on the constant
scalar curvature beyond the values allowed by the soliton equation.
In the splitting, the Einstein or Euclidean factor may have dimension zero.

\begin{remark}\label{rem:total-geodesy}
Total geodesy is used to eliminate the first normal derivative of the
Ricci projection defect. The higher-order argument requires neither
flatness of the normal connection nor triviality of the normal bundle.
\end{remark}

The proof compares two estimates for an averaged curvature term on the
normal spheres. Their difference excludes every normal degree at least
two. Real analyticity then turns this local vanishing into the identity
\begin{equation}\label{eq:ricci-projection}
 (\Ric^\sharp)^2=\frac12\Ric^\sharp,
\end{equation}
which gives constant Ricci rank and hence rigidity.

The paper is organized as follows. Section~\ref{sec:preliminaries}
collects the basic soliton identities and level-set formulas.
Sections~\ref{sec:focal}--\ref{sec:first} establish the focal structure,
the Ricci projection defect, and its first normal derivative.
Sections~\ref{sec:balance}--\ref{sec:higher} develop the scalar balance
and spherical estimates and exclude all higher normal coefficients.
Section~\ref{sec:splitting} completes the proof by analytic continuation
and establishes Corollary~\ref{cor:cao-general}.

\section{Preliminaries}\label{sec:preliminaries}

In this section we recall the standard soliton identities, regular
level-set formulas, and known rigidity results used below. We state
these facts in our notation without repeating their proofs.

Throughout the paper, $R$ denotes scalar curvature and $\Delta=\tr\nabla^2$.
Our curvature convention is
\[
 \Rm(X,Y)Z=\nabla_X\nabla_YZ-\nabla_Y\nabla_XZ-\nabla_{[X,Y]}Z,
\]
so that $\langle\Rm(X,Y)Y,X\rangle$ is positive on the unit sphere for
orthonormal $X,Y$. We use the weighted Laplacian
$\Delta_f=\Delta-\nabla_{\nabla f}$, with covariant differentiation
when it acts on tensors.

The following identities are standard; see~\cite{CaoSurvey,EminentiLaNaveMantegazza,PetersenWylie}.
\begin{lemma}[Basic soliton identities]\label{lem:basic-identities}
A solution of \eqref{eq:soliton} satisfies
\begin{align}
 R+\Delta f&=\frac n2,\label{eq:trace-soliton}\\
 dR&=2\Ric(\nabla f,\cdot),\label{eq:gradient-scalar}\\
 R+|\nabla f|^2-f&=\mathrm{constant},\label{eq:hamilton-identity}\\
 \Delta_fR&=R-2|\Ric|^2.\label{eq:scalar-drift}
\end{align}
Moreover, for vector fields $X,Y,Z$,
\begin{equation}\label{eq:soliton-codazzi}
 \begin{aligned}
 (\nabla_X\Ric)(Y,Z)-(\nabla_Y\Ric)(X,Z)
 =-\langle\Rm(X,Y)\nabla f,Z\rangle.
 \end{aligned}
\end{equation}
In particular, $\nabla\Ric$ is fully symmetric at every critical point
of $f$.
\end{lemma}

We now assume $n=k+m$ and $R=k/2$, as in the main theorems. The following
normalization and identities are the constant-scalar-curvature
specialization of Lemma~\ref{lem:basic-identities};
see also~\cite{FernandezLopezGarciaRio}.

\begin{corollary}\label{cor:normalized-identities}
After adding a constant to $f$, one has
\begin{equation}\label{eq:normalization}
 |\nabla f|^2=f\ge0.
\end{equation}
Furthermore,
\begin{equation}\label{eq:standard-identities}
 \Ric(\nabla f,\cdot)=0,\qquad |\Ric|^2=\frac k4,
\end{equation}
\begin{equation}\label{eq:lap-f}
 \Delta f=\frac m2,\qquad \Delta_ff=\frac m2-f.
\end{equation}
\end{corollary}

For endomorphism calculations we write $A=\Ric^\sharp$, that is,
$g(AX,Y)=\Ric(X,Y)$. Thus
\begin{equation}\label{eq:A-identities}
 A\nabla f=0,\qquad \tr A=\frac k2,\qquad |A|^2=\frac k4.
\end{equation}
On $\{f>0\}$, put $r=2\sqrt f$, $\nu=\nabla r$, and $p=m-1$.
We denote a regular level set of $r$ by $\Sigma_r$ and its shape
operator with respect to $\nu$ by $S$.

We record the standard level-set formulas in the radial variable $r$.
They follow from the soliton, Gauss, and Riccati equations;
see~\cite{Besse,PetersenWylie,LiQuWu}.
\begin{lemma}[Regular level sets]\label{lem:level-geometry}
On $\{f>0\}$ one has
\begin{equation}\label{eq:radial-identities}
 |\nabla r|=1,\qquad A\nu=0,\qquad
 \nabla_\nu\nu=0,\qquad \Delta r=\frac pr.
\end{equation}
On $T\Sigma_r$, the shape operator satisfies
\begin{equation}\label{eq:shape-level-set}
 S=\bigl(\restr{\nabla^2r}{T\Sigma_r}\bigr)^\sharp
   =\frac1r(\Id-2A),\qquad
 \tr S=\frac pr,\qquad |S|^2=\frac p{r^2}.
\end{equation}
Moreover, the induced Ricci endomorphism is
\begin{equation}\label{eq:ric-level-set}
 \Ric_{\Sigma_r}^\sharp
 =A-\frac2{r^2}\nabla_{r\nu}A
   +\frac{p-1}{r^2}(\Id-2A).
\end{equation}
\end{lemma}

\begin{remark}\label{rem:level-scalar}
Tracing the Gauss equation also gives
\[
 R_{\Sigma_r}=R-2\Ric(\nu,\nu)+(\tr S)^2-|S|^2
 =\frac k2+\frac{p(p-1)}{r^2}.
\]
Thus the scalar curvature of the full level set is constant along each
level. The scalar curvature of an individual normal sphere requires the
additional fiber calculation in Section~\ref{sec:balance}.
\end{remark}

We also use the properness of the potential and its compact minimum
set, following Cao--Zhou~\cite{CaoZhou}, and the local real analyticity
of the soliton metric, following Kotschwar~\cite{Kotschwar}.
\begin{lemma}\label{lem:proper-analytic}
Under the hypotheses of Theorem~\ref{thm:automatic-focal}, $M$ is
noncompact, $f$ is proper, and $M_-=f^{-1}(0)$ is nonempty and compact.
In addition, $g$ is locally real analytic in suitable coordinates.
\end{lemma}

Finally, the following known classification facts identify the range
that remains to be treated by the focal-set argument.

\begin{lemma}[Known scalar-curvature cases]\label{lem:known-cases}
Let $(M^n,g,f)$ be a complete shrinking gradient Ricci soliton satisfying
\eqref{eq:soliton}, with constant scalar curvature $R$. Then
\[
 R\in\left\{0,\frac12,\ldots,\frac{n-1}{2},\frac n2\right\}.
\]
For $n\ge2$, the value $R=1/2$ does not occur. The cases
$R=0$, $R=(n-1)/2$, and $R=n/2$ are rigid. In the cases $R=0$ and
$R=n/2$, the universal cover is Gaussian and Einstein, respectively.
\end{lemma}

The discreteness of $R$, the exclusion of $R=1/2$, and the case
$R=(n-1)/2$ are due to
Fern\'andez-L\'opez--Garc\'ia-R\'io~\cite{FernandezLopezGarciaRio}.
The extremal cases are included in Petersen--Wylie~\cite{PetersenWylie}.
We will also use the constant-rank criterion
of~\cite{FernandezLopezGarciaRio}: a complete shrinking gradient Ricci
soliton with constant scalar curvature and constant Ricci rank is rigid.

\section{The automatic global focal structure}\label{sec:focal}

We prove Theorem~\ref{thm:automatic-focal}. The argument separates the local transnormal information at the minimum set from a global eikonal argument that rules out a cut locus.

\begin{proof}[Proof of Theorem~\ref{thm:automatic-focal}]
By Lemma~\ref{lem:proper-analytic}, $M_-$ is nonempty and compact.
We first determine its smooth structure and the Hessian of $f$ along it.

	The identities $|\nabla f|^2=f$ and $\Delta f=m/2$ show that $f$ is isoparametric and in particular transnormal. Wang's focal-variety theorem for a transnormal function on a complete connected manifold implies that the nonempty minimum focal variety $M_-$ is a smooth submanifold~\cite{Wang}. For $b(t)=t$ in $|\nabla f|^2=b(f)$, the Hessian at the minimum variety satisfies
	\begin{equation}
		\restr{\nabla^2 f}{TM_-}=0,
		\qquad
		\restr{\nabla^2 f}{\nu M_-}=\frac{b'(0)}2g=\frac12g,
		\qquad
		\nabla^2 f(TM_-,\nu M_-)=0.\label{eq:hess-on-focal}
	\end{equation}
	Taking the trace and using \eqref{eq:lap-f} yields $\codim M_-=m$, hence $\dim M_-=k$. The soliton equation then gives
	\begin{equation}
		\restr{\Ric}{TM_-}=\frac12\restr{g}{TM_-},
		\qquad
		\restr{\Ric}{\nu M_-}=0,
		\qquad
		\Ric(TM_-,\nu M_-)=0,\label{eq:ric-on-focal}
	\end{equation}
	which is exactly $\restr{\Ric^\sharp}{M_-}=\frac12\Pi_{TM_-}$.
	
	We now justify the global normal coordinates carefully. On $M\setminus M_-$ define
	\[
	r=2\sqrt f.
	\]
	Then \eqref{eq:normalization} gives
	\begin{equation}
		|\nabla r|^2=1.\label{eq:eikonal}
	\end{equation}
	Differentiating \eqref{eq:eikonal} gives $\nabla^2 r(\nabla r,\cdot)=0$. Hence the integral curves of $\pm\nabla r$ are unit-speed geodesics, and $r$ changes on them with speed $\pm1$.
	
	Fix $x\in M\setminus M_-$ and let $\gamma$ be the geodesic with $\gamma(0)=x$, $\dot\gamma(0)=-\nabla r(x)$. For $0\le t<r(x)$,
	\[
	r(\gamma(t))=r(x)-t>0,
	\]
	so the curve remains in $M\setminus M_-$. Completeness extends $\gamma$ to $t=r(x)$, and continuity gives $q:=\gamma(r(x))\in M_-$. Thus the segment has length $r(x)$ and
	\begin{equation}
		d(x,M_-)\le r(x).\label{eq:distance-upper}
	\end{equation}
	Conversely, if $c$ is any piecewise smooth curve from $M_-$ to $x$, apply $|\nabla r|=1$ to the portion where $r\ge\varepsilon$ and let $\varepsilon\downarrow0$ to obtain $r(x)\le L(c)$. Taking the infimum gives
	\begin{equation}
		r=d(M_-,\cdot).\label{eq:distance-equality}
	\end{equation}
	
	It remains to show that $\Exp^\perp$ has neither identifications nor focal points. By the tubular-neighborhood theorem, for $q\in M_-$ and unit $\theta\in\nu_qM_-$, the normal geodesic $\sigma(t)=\exp_q(t\theta)$ satisfies $r(\sigma(t))=t$ and $\dot\sigma(t)=\nabla r$ for small $t>0$. Both $\sigma$ and the integral curves of $\nabla r$ solve the geodesic equation. If there were a first finite time at which either identity failed, then at that time $r=t>0$ and $\nabla r$ would still be smooth; uniqueness of the geodesic ODE would extend the identities past that time. Therefore
	\begin{equation}
		r(\exp_q(t\theta))=t,
		\qquad
		\frac d{dt}\exp_q(t\theta)=\nabla r
		\qquad(t>0).\label{eq:normal-geodesic}
	\end{equation}
	By \eqref{eq:distance-equality}, every such normal segment minimizes the distance to $M_-$ for every $t>0$.
	
	Surjectivity of $\Exp^\perp$ follows from the backward $-\nabla r$ geodesic constructed above; the first-variation formula shows that its terminal segment is orthogonal to $M_-$. For injectivity, if
	\[
	\exp_q(v)=\exp_{q'}(v')=x,
	\]
	then \eqref{eq:normal-geodesic} gives $|v|=|v'|=r(x)$ and both terminal velocities are $\nabla r(x)$. A geodesic is uniquely determined by terminal position and velocity, hence the two normal geodesics coincide and $(q,v)=(q',v')$.
	
	Finally, $d\Exp^\perp$ cannot be singular at a nonzero normal vector. A singularity of the normal exponential map is a focal point of $M_-$; by the standard index-form characterization, a normal geodesic ceases to be minimizing beyond its first focal point. This contradicts the global minimizing property just obtained. Along the zero section, $d\Exp^\perp$ is the canonical isomorphism $TM_-\oplus\nu M_-\simeq TM|_{M_-}$. Hence $\Exp^\perp$ is a bijective local diffeomorphism, and therefore a global diffeomorphism.
	
	The total space $\nu M_-$ has one connected component for each connected component of $M_-$. Since $M$ is connected and $\Exp^\perp$ is a diffeomorphism, $M_-$ is connected.
\end{proof}

\section{The Ricci projection defect}\label{sec:defect}

We next compare the Ricci endomorphism with the orthogonal projection
onto the directions tangent to the minimum set. The resulting tensor
has an exact square identity, which will be used in place of a curvature
sign assumption. We also derive the volume identities needed when
integrating over the normal spheres.

We use the endomorphism $A=\Ric^\sharp$ and the identities
\eqref{eq:A-identities} from Section~\ref{sec:preliminaries}.
For an orthogonal distribution, a subscript on $A$ denotes its
compression to that distribution; it does not assert invariance under $A$.

Use the normal exponential diffeomorphism to define $\pi:M\to M_-$ and put
\begin{equation}
	\calN=\ker d\pi,
	\qquad
	\calH=\calN^\perp.\label{eq:HN-distributions}
\end{equation}
Both distributions extend smoothly across the zero section. On $M\setminus M_-$,
\[
\calN=\R\nu\oplus\calV,
\qquad
\nu=\nabla r,
\qquad
\dim\calV=p=m-1.
\]
Define
\begin{equation}
	E=A-\frac12\Pi_{\calH},
	\qquad
	v=\tr_{\calN}A=\tr_{\calV}A,
	\qquad
	Z=\restr{\Pi_{\calH}A}{\calN}.\label{eq:defect-definition}
\end{equation}
Since $A\nu=0$, the mixed block $Z$ vanishes on the radial direction.

\begin{lemma}[Exact square]\label{lem:exact-square}
	On the whole normal tube,
	\begin{equation}
		|E|^2=v=|A_{\calN}|^2+2|Z|^2+\left|A_{\calH}-\frac12\Id_{\calH}\right|^2.\label{eq:exact-square}
	\end{equation}
	In particular $v\ge0$ without any sign assumption on $\Ric$.
\end{lemma}

\begin{proof}
Since $A$ is self-adjoint, its block matrix in $\calH\oplus\calN$ is
\[
 A=\begin{pmatrix}A_{\calH}&Z\\Z^*&A_{\calN}\end{pmatrix},
 \qquad
 E=\begin{pmatrix}A_{\calH}-\frac12\Id_{\calH}&Z\\
                  Z^*&A_{\calN}\end{pmatrix}.
\]
The two off-diagonal blocks have the same squared norm. Therefore
\[
 |E|^2=\left|A_{\calH}-\frac12\Id_{\calH}\right|^2
        +2|Z|^2+|A_{\calN}|^2.
\]
On the other hand, $|\Pi_{\calH}|^2=k$ and
$\langle A,\Pi_{\calH}\rangle=\tr_{\calH}A$. Hence
\[
\begin{aligned}
 |E|^2
 &=|A|^2-\langle A,\Pi_{\calH}\rangle+\frac14|\Pi_{\calH}|^2\\
 &=\frac k4-\tr_{\calH}A+\frac k4
  =\tr A-\tr_{\calH}A=\tr_{\calN}A=v.
\end{aligned}
\]
Combining the two expressions proves the identity and $v\ge0$.
\end{proof}

\begin{remark}\label{rem:blocks}
The horizontal and normal distributions need not be invariant under $A$;
the tensor $Z$ records their mixed block. All norms of endomorphisms are
Hilbert--Schmidt norms. The musical isomorphisms $\sharp$ and $\flat$
raise and lower an index with the metric in use.
\end{remark}

For $x\in M_-$, let
\[
S_{x,r}=\exp_x\bigl(rS(\nu_xM_-)\bigr)
\]
be the actual normal sphere of radius $r$, and let
\begin{equation}
	h_r=r^{-2}\restr{g}{TS_{x,r}}.\label{eq:normalized-sphere-metric}
\end{equation}
We regard $h_r$ as a metric on the fixed unit sphere $S(\nu_xM_-)$ by
pullback under $\theta\mapsto\exp_x(r\theta)$. Write $g_0$ for its unit
round metric and the normalized fiber measure as
\begin{equation}
	\frac{dV_{h_r}}{\omega_p}=\rho(x,r\theta)\dd\omega(\theta),\label{eq:fiber-density}
\end{equation}
where $d\omega$ is probability measure on the unit round sphere and $\omega_p=\Vol(\Sp,g_0)$.

\begin{lemma}[Angular evolution and exact volume law]\label{lem:angular-evolution}
	Let $s=\log r$. Then
	\begin{equation}
		\partial_s h_r=-4h_rA_{\calV},
		\qquad
		\partial_s\log\rho=-2v,
		\qquad
		\rho(x,0)=1.\label{eq:angular-evolution}
	\end{equation}
	Moreover
	\begin{equation}
		dV_g=\omega_pr^p\dd r\dd V_{M_-}\dd\omega.\label{eq:exact-volume-law}
	\end{equation}
\end{lemma}

\begin{proof}
Let $J_1,J_2$ be tangent vectors obtained by varying the angular
coordinates while keeping $x$ fixed. They commute with $\nu$, so
\[
 \partial_rg(J_1,J_2)
 =g(\nabla_{J_1}\nu,J_2)+g(J_1,\nabla_{J_2}\nu)
 =2g(SJ_1,J_2).
\]
Using \eqref{eq:shape-level-set} and $h_r=r^{-2}g|_{\calV}$, we obtain
\[
\begin{aligned}
 r\partial_rh_r
 &=-2r^{-2}g|_{\calV}+r^{-1}\partial_r(g|_{\calV})\\
 &=-2h_r+2h_r(\Id-2A_{\calV})=-4h_rA_{\calV}.
\end{aligned}
\]
For a varying metric, $\partial_s\log dV_{h_r}
=\frac12\tr(h_r^{-1}\partial_sh_r)$. Consequently,
\[
 \partial_s\log\rho
 =\frac12\tr\bigl(h_r^{-1}(-4h_rA_{\calV})\bigr)
 =-2\tr_{\calV}A=-2v.
\]
The initial value $\rho(x,0)=1$ follows from the Euclidean leading term
of the normal exponential map.

For the full Jacobian, write the pulled-back volume density locally as
\[
 dV_g=J(x,r,\theta)\dd r\dd V_{M_-}\dd V_{g_0}.
\]
The variation of the level-set measure and \eqref{eq:radial-identities}
give
\[
 \partial_r\log J=\tr S=\frac pr,
 \qquad \partial_r\log\left(\frac J{r^p}\right)=0.
\]
Since $J/r^p\to1$ as $r\downarrow0$, we have $J=r^p$.
Finally, $dV_{g_0}=\omega_p\dd\omega$, which proves
\eqref{eq:exact-volume-law}.
\end{proof}

\section{The first normal jet}\label{sec:first}

The square identity alone does not determine how rapidly the defect
vanishes near $M_-$. In this section we compute its first normal derivative
and identify the precise contribution of the second fundamental form.
This is the step where the total-geodesy assumption enters the proof.

For $\xi\in\nu_xM_-$, let $L_\xi:T_xM_-\to T_xM_-$ be the shape operator of $M_-$,
\begin{equation}
	\langle L_\xi X,Y\rangle=\langle\II_{M_-}(X,Y),\xi\rangle.\label{eq:shape-focal}
\end{equation}

\begin{lemma}[First normal jet]\label{lem:first-normal-jet}
	At every $x\in M_-$ and $\xi\in\nu_xM_-$,
	\begin{equation}
		\restr{(\nabla_\xi A)}{T_xM_-}=\frac12L_\xi,\label{eq:first-jet-A}
	\end{equation}
	and every first derivative component of $A$ containing at least two normal tensor slots vanishes. Consequently,
	\begin{equation}
		E(x,r\theta)=\frac r2
		\begin{pmatrix}
			L_\theta&0\\
			0&0
		\end{pmatrix}
		+O(r^2).\label{eq:first-jet-E}
	\end{equation}
\end{lemma}

\begin{proof}
At $M_-$, $df=0$. By \eqref{eq:soliton-codazzi}, $\nabla\Ric$ is
fully symmetric there. We compute its components according to the
splitting $T_xM=T_xM_-\oplus\nu_xM_-$.

	Let $X,Y\in T_xM_-$ and $\xi\in\nu_xM_-$. Since $\Ric(\xi,Y)=0$ along $M_-$, differentiation in a tangential direction $X$ gives
	\[
	0=(\nabla_X\Ric)(\xi,Y)+\Ric(\nabla_X\xi,Y)+\Ric(\xi,\nabla_XY).
	\]
	The last term vanishes by \eqref{eq:ric-on-focal}, and the tangential component of $\nabla_X\xi$ is $-L_\xi X$. Hence
	\[
	(\nabla_X\Ric)(\xi,Y)=\frac12\langle L_\xi X,Y\rangle.
	\]
	Full symmetry of $\nabla\Ric$ moves the normal derivative to the first slot and gives \eqref{eq:first-jet-A}.
	
	Next, the normal-normal block $A|_{\nu M_-}$ vanishes identically along $M_-$. Differentiating this identity tangentially shows that every component of $\nabla A$ with two normal slots and one tangent slot vanishes. Along a normal geodesic $\gamma$, $A\dot\gamma=0$ because $A\nabla f=0$ and $\dot\gamma$ is parallel to $\nabla f$ off $M_-$. Since $\dot\gamma$ is parallel, differentiation gives $(\nabla_{\dot\gamma}A)\dot\gamma=0$. For normal vectors $\xi,\eta$, this says
\[
 (\nabla_\xi\Ric)(\xi,\eta)=0.
\]
Replacing $\xi$ by $\xi+\zeta$ and using the full symmetry already
proved gives
\[
2(\nabla_\xi\Ric)(\zeta,\eta)=0.
\]
Thus every all-normal component also vanishes.
	
Finally, use parallel transport along $\gamma(t)=\exp_x(t\theta)$
to compare the tensors with their values at $x$. A variation of the
initial normal velocity in a direction $\eta\perp\theta$ gives a Jacobi
field with
\[
 J_\eta(0)=0,\qquad J_\eta'(0)=\eta,\qquad
 J_\eta(t)=t\eta+O(t^3).
\]
Together with the radial vector, these fields span $\calN$. Hence, in
this parallel identification,
\[
 \Pi_{\calN}(t)=\Pi_{\nu_xM_-}+O(t^2),\qquad
 \Pi_{\calH}(t)=\Pi_{T_xM_-}+O(t^2).
\]
In particular, the normal derivative of $\Pi_{\calH}$ at $M_-$ is zero.
Taylor expansion of $A-\frac12\Pi_{\calH}$ now gives
\eqref{eq:first-jet-E}.
\end{proof}

\begin{corollary}\label{cor:second-order-vanishing}
	If $M_-$ is totally geodesic, then
	\begin{equation}
		E=O(r^2)\label{eq:E-order-two}
	\end{equation}
	uniformly on the compact unit normal sphere bundle.
\end{corollary}

\begin{proof}
When $\II_{M_-}=0$, every $L_\theta$ is zero, so
\eqref{eq:first-jet-E} gives the assertion. Compactness of $M_-$ makes the
Taylor remainder uniform over its unit normal sphere bundle.
\end{proof}

\section{A scalar balance in arbitrary codimension}\label{sec:balance}

We now relate the defect to the intrinsic scalar curvature of each normal
sphere. The calculation has two parts: the soliton equation determines
the curvature of the level set, and the Gauss equation compares that
curvature with the curvature of the fiber. We keep the resulting
integrated identity in a form suitable for comparison with the spherical
estimate in the next section.

Assume from now on that $M_-$ is totally geodesic. Suppose, for contradiction, that $E$ is not flat to infinite normal order. Let $d\ge2$ be the smallest normal degree for which the Taylor coefficient is nonzero:
\begin{equation}
	E(x,r\theta)=r^dE_d(x,\theta)+O(r^{d+1}).\label{eq:leading-E}
\end{equation}
Relative to $\calH\oplus\calV$ write
\begin{equation}
	E_d=
	\begin{pmatrix}
		H_d&Z_d\\
		Z_d^*&N_d
	\end{pmatrix}.\label{eq:block-Ed}
\end{equation}
All expansions are taken in smooth local trivializations over $M_-$ and
are uniform, with the derivatives used below, on its compact unit normal
sphere bundle. Since all lower normal coefficients vanish, the leading
coefficient is unchanged by a smooth change of trivialization that is
the identity at $M_-$. In the limiting splitting $T_xM_-\oplus\nu_xM_-$, its
components are homogeneous polynomials of degree $d$ in the normal
variable. The identity $E\nu=0$ implies $N_d\theta=0$ and $Z_d\theta=0$.
Consequently the normal block, extended by zero in the radial direction,
still has polynomial ambient components of degree $d$.

Use the probability measure
\[
d\mu_0=\frac{dV_{M_-}}{\Vol(M_-)}\dd\omega
\]
on $S(\nu M_-)$ and set
\begin{equation}
	\eH=\int|H_d|^2\dd\mu_0,
	\qquad
	\eN=\int|N_d|^2\dd\mu_0,
	\qquad
	\eZ=\int|Z_d|^2\dd\mu_0,\label{eq:H-N-Z}
\end{equation}
\begin{equation}
	\eB=\eH+\eN+2\eZ>0.\label{eq:B-energy}
\end{equation}

\begin{lemma}[Leading density expansion]\label{lem:density-expansion}
	One has
	\begin{equation}
		v=r^{2d}|E_d|^2+O(r^{2d+1}),\label{eq:v-expansion}
	\end{equation}
	\begin{equation}
		\rho=1-\frac{r^{2d}}d|E_d|^2+O(r^{2d+1}),\label{eq:rho-expansion}
	\end{equation}
	and
	\begin{equation}
		\tr_{\Sp}N_d=0.\label{eq:Nd-tracefree}
	\end{equation}
\end{lemma}

\begin{proof}
The exact square identity and \eqref{eq:leading-E} give
\[
 v=\left|r^dE_d+O(r^{d+1})\right|^2
   =r^{2d}|E_d|^2+O(r^{2d+1}).
\]
Since $r\partial_r\log\rho=-2v$ and $\rho(x,0)=1$, integration along
a fixed ray gives
\[
\begin{aligned}
 \log\rho(x,r\theta)
 &=-2\int_0^r\frac{v(x,t\theta)}t\dd t\\
 &=-2|E_d(x,\theta)|^2\int_0^rt^{2d-1}\dd t+O(r^{2d+1})\\
 &=-\frac{r^{2d}}d|E_d(x,\theta)|^2+O(r^{2d+1}).
\end{aligned}
\]
Exponentiating introduces an additional error $O(r^{4d})$, which is
absorbed by $O(r^{2d+1})$. This proves \eqref{eq:rho-expansion}.

Finally, $A_{\calV}=E_{\calV}$, so
\[
 \tr_{\calV}A=r^d\tr_{g_0}N_d+O(r^{d+1}).
\]
The left-hand side is $v=O(r^{2d})$. Dividing by $r^d$ and letting
$r\downarrow0$ proves \eqref{eq:Nd-tracefree}.
\end{proof}

We use the level-set curvature formula \eqref{eq:ric-level-set} to
compute the scalar curvature of a normal fiber. The difference between
these two curvatures is recorded by the following geometric terms.

Inside $\Sigma_r$, let $B_{\calV}$ and $H_{\calV}$ denote the second fundamental form and mean curvature of the integrable vertical distribution $\calV$. Let
\[
\calM=\sum_{a=1}^p\sum_{\alpha=1}^k\secop_{\Sigma_r}(e_a,e_\alpha)
\]
be the mixed sectional-curvature sum, and define
\begin{equation}
	G=|H_{\calV}|^2-|B_{\calV}|^2-\calM.\label{eq:G-definition}
\end{equation}

\begin{lemma}[General-codimension fiber scalar identity]\label{lem:fiber-scalar}
	Let $R_{h_r}$ be the scalar curvature of the normalized normal-sphere metric $h_r$. Then
	\begin{equation}
		R_{h_r}-p(p-1)=\bigl(r^2-2(p-1)\bigr)v-2v_s-8|Z|^2+r^2G.\label{eq:fiber-scalar-identity}
	\end{equation}
\end{lemma}

\begin{proof}
We use \eqref{eq:shape-level-set} to derive the radial form of the
level-set Ricci tensor before taking its vertical trace. Radial
derivatives of endomorphisms are understood in a parallel frame.

Combine the hypersurface Gauss
and Riccati equations:
\[
\begin{aligned}
 \Ric_{\Sigma_r}^\sharp
 &=A-\Rm(\cdot,\nu)\nu+(\tr S)S-S^2,\\
 \Rm(\cdot,\nu)\nu&=-\nabla_\nu S-S^2.
\end{aligned}
\]
It follows that $\Ric_{\Sigma_r}^\sharp=A+\nabla_\nu S+(\tr S)S$.
Differentiating $S=r^{-1}(\Id-2A)$ in a radial parallel frame gives
\[
 \nabla_\nu S=-\frac1{r^2}(\Id-2A)-\frac2r\nabla_\nu A,
 \qquad (\tr S)S=\frac p{r^2}(\Id-2A).
\]
Adding the two expressions proves \eqref{eq:ric-level-set}.

 The vertical trace requires care because $\calV$ moves with $r$.
 Let $\Pi_{\calV}$ denote its orthogonal projection. A vertical Jacobi
 field $J$ satisfies $[\nu,J]=0$, hence
 $\nabla_{r\nu}J=rSJ=(\Id-2A)J$. Its horizontal component is therefore
 $-2ZJ$. Differentiating the projection gives, in the order
 $\calH\oplus\calV$,
 \[
 \nabla_{r\nu}\Pi_{\calV}
 =\begin{pmatrix}0&-2Z\\-2Z^*&0\end{pmatrix}.
 \]
The two diagonal blocks of
$\nabla_{r\nu}\Pi_{\calV}$ vanish by differentiating
$\Pi_{\calV}^2=\Pi_{\calV}$; the lower-left block is the adjoint of
the upper-right block. Consequently,
\[
 \tr\bigl(A\nabla_{r\nu}\Pi_{\calV}\bigr)
 =-2\tr(ZZ^*)-2\tr(Z^*Z)=-4|Z|^2.
\]
Taking the derivative of $v=\tr(A\Pi_{\calV})$ therefore yields
 \begin{equation}\label{eq:vertical-derivative}
 v_s=\tr_{\calV}(\nabla_{r\nu}A)-4|Z|^2.
 \end{equation}
 Since $\tr_{\calV}(\Id-2A)=p-2v$, the vertical trace of
 \eqref{eq:ric-level-set} becomes
 \begin{equation}\label{eq:vertical-ricci}
 \begin{aligned}
 r^2\tr_{\calV}\Ric_{\Sigma_r}^{\sharp}
 &=r^2v-2v_s-8|Z|^2+(p-1)(p-2v)\\
 &=p(p-1)+\bigl(r^2-2(p-1)\bigr)v-2v_s-8|Z|^2.
 \end{aligned}
 \end{equation}

Finally, choose orthonormal frames $e_a$ of $\calV$ and $e_\alpha$
of $\calH$. The vertical trace of the Ricci tensor is
\[
 \tr_{\calV}\Ric_{\Sigma_r}^\sharp
 =\sum_{a,b}\langle\Rm_{\Sigma_r}(e_a,e_b)e_b,e_a\rangle
   +\calM.
\]
The Gauss equation for $S_{x,r}\subset\Sigma_r$ gives
\[
 R_{S_{x,r}}
 =\sum_{a,b}\langle\Rm_{\Sigma_r}(e_a,e_b)e_b,e_a\rangle
   +|H_{\calV}|^2-|B_{\calV}|^2.
\]
Since $g|_{S_{x,r}}=r^2h_r$, scalar curvature rescales by $r^{-2}$.
Thus
\[
 r^{-2}R_{h_r}=R_{S_{x,r}}
             =\tr_{\calV}\Ric_{\Sigma_r}^\sharp+G.
\]
	Multiplying by $r^2$ and using \eqref{eq:vertical-ricci} yields \eqref{eq:fiber-scalar-identity}.
\end{proof}

To integrate this identity, it is enough to introduce two scalar quantities:
\begin{equation}\label{eq:integrated-quantities}
 \calS(r)=\int\bigl(R_{h_r}-p(p-1)\bigr)\rho\dd\mu_0,
 \qquad \Gamma(r)=\int G\rho\dd\mu_0.
\end{equation}
Since $\rho_s=-2v\rho$, differentiation under the integral gives
\[
 \frac{d}{ds}\int v\rho\dd\mu_0
 =\int(v_s-2v^2)\rho\dd\mu_0.
\]
In particular,
\[
 -2\int v_s\rho\dd\mu_0
 =-2\frac d{ds}\int v\rho\dd\mu_0-4\int v^2\rho\dd\mu_0.
\]
Substituting this into the integral of
\eqref{eq:fiber-scalar-identity} yields the exact balance
\begin{equation}\label{eq:exact-balance}
\begin{aligned}
 \calS(r)={}&\bigl(r^2-2(p-1)\bigr)\int v\rho\dd\mu_0
 -2\frac{d}{ds}\int v\rho\dd\mu_0\\
 &-4\int v^2\rho\dd\mu_0-8\int|Z|^2\rho\dd\mu_0+r^2\Gamma(r).
\end{aligned}
\end{equation}
The weight $\rho$ is the actual fiber-volume density. It is different
from the unweighted measure in the full level-set volume law, a distinction
that will be essential in Section~\ref{sec:higher}.

\section{A sharp spherical quadratic estimate}\label{sec:sphere}

The preceding balance contains the scalar curvature of a metric close
to the round sphere. We now estimate its first nonzero variation.
The estimate depends only on the degree of the polynomial leading term;
no decomposition into tensor spherical harmonics is needed.

For a metric $h$ on $\Sp$, let
\[
V(h)=\frac{\Vol(\Sp,h)}{\omega_p},
\qquad
I(h)=\frac1{\omega_p}\int_{\Sp}R_h\dd V_h.
\]
For $p\ge2$, define
\begin{equation}
	Y_p(h)=V(h)^{-(p-2)/p}I(h).\label{eq:Y-functional}
\end{equation}
Then $Y_p(g_0)=p(p-1)$. For $p=1$, set $Y_1\equiv0$.

\begin{samepage}
\begin{lemma}[Spherical Einstein--Hilbert estimate]\label{lem:spherical-EH}
	Let $K$ be a smooth trace-free symmetric two-tensor tangent to the unit sphere $\Sp$. Assume that, after extending $K$ by zero on the radial direction, every ambient Cartesian component is the restriction of a polynomial of degree at most $d$. If
	\[
	h_\varepsilon=g_0+\varepsilon K+O(\varepsilon^2)
	\]
	and
	\[
	Y_p(h_\varepsilon)=p(p-1)+\varepsilon^2y(K)+O(\varepsilon^3),
	\]
	then
	\begin{equation}
		y(K)\ge-\frac14d(d+p-1)|K|_{L^2(\Sp)}^2.\label{eq:spherical-EH-estimate}
	\end{equation}
	The $L^2$ norm is normalized by $\omega_p^{-1}$.
\end{lemma}
\end{samepage}

\begin{proof}
	For $p=1$ the assertion is the definition, and for $p=2$ the Einstein--Hilbert functional is topological, so $y(K)=0$. Assume $p\ge3$.
	
We give the second variation directly. Put
 $(\operatorname{div}K)_j=\nabla^iK_{ij}$ and let
 $\nabla^*\nabla$ be the positive rough Laplacian. The first variation of
 the volume-normalized total scalar curvature vanishes at the round
 Einstein metric. Therefore its second variation depends only on $K$,
 not on the second derivative of the chosen path.

The volume and total-scalar-curvature variations are
\[
 DV|_h(K)=\frac1{2\omega_p}\int\tr_hK\dd V_h,
 \qquad
 DI|_h(K)=-\frac1{\omega_p}\int
 \left\langle\Ric_h-\frac12R_hh,K\right\rangle_h\dd V_h.
\]
In the second formula the divergence terms have zero integral on the
closed sphere. Differentiating $Y_p=V^{-(p-2)/p}I$ therefore gives
 \[
 DY_p|_h(K)=-V(h)^{-(p-2)/p}\frac1{\omega_p}
 \int\left\langle\Ric_h-\frac12R_hh
       +\frac{p-2}{2p}\frac{I(h)}{V(h)}h,K\right\rangle_h\dd V_h.
 \]
 At $g_0$, the tensor in brackets is zero. Since $\tr_{g_0}K=0$,
 differentiating this formula gives
 \[
 D^2Y_p|_{g_0}(K,K)
 =-\int\langle\Ric'(K)-(p-1)K,K\rangle\dd\omega.
 \]
 Terms proportional to $g_0$ have zero pairing with $K$.
To compute $\Ric'(K)$, the variation of the connection in local
coordinates is
\[
 (\Gamma^\ell_{ij})'
 =\frac12g^{\ell a}
   \bigl(\nabla_iK_{ja}+\nabla_jK_{ia}-\nabla_aK_{ij}\bigr).
\]
Differentiating curvature and contracting gives
\[
 \Ric'(K)_{ij}
 =\nabla_\ell(\Gamma^\ell_{ij})'
  -\nabla_j(\Gamma^\ell_{i\ell})'.
\]
Substitution of the connection variation, commutation of covariant
derivatives, and $\tr K=0$ give the standard formula~\cite{Besse}
 \[
 \Ric'(K)=\frac12\left(\Delta_LK+
       \mathcal L_{(\operatorname{div}K)^\sharp}g_0\right),
 \]
 where
 \[
 \Delta_LK=\nabla^*\nabla K+\Ric\circ K+K\circ\Ric-2\mathring RK.
 \]
 On the unit sphere $\mathring RK=(\tr K)g_0-K=-K$, so
 \begin{equation}\label{eq:Lichnerowicz-sphere}
 \Delta_LK-2(p-1)K=\nabla^*\nabla K+2K.
 \end{equation}
 Integration by parts gives
 \[
 \int\langle\mathcal L_{(\operatorname{div}K)^\sharp}g_0,K\rangle
       \dd\omega=-2\int|\operatorname{div}K|^2\dd\omega.
 \]
 As $y(K)=\frac12D^2Y_p|_{g_0}(K,K)$, we conclude that
 \begin{equation}\label{eq:y-lower-gradient}
 \begin{aligned}
 y(K)&=-\frac14\int\bigl(|\nabla K|^2+2|K|^2\bigr)\dd\omega
       +\frac12\int|\operatorname{div}K|^2\dd\omega\\
 &\ge-\frac14\int\bigl(|\nabla K|^2+2|K|^2\bigr)\dd\omega.
 \end{aligned}
 \end{equation}

	To compare the right-hand side with scalar spherical harmonics, regard $K$ as an ambient symmetric bilinear form with $K(\theta,\cdot)=0$. Let $D$ be the Euclidean connection and $e_1,\ldots,e_p$ an orthonormal tangent frame. The Gauss formula gives
	\begin{equation}
		\sum_{a=1}^p|D_{e_a}K|^2=|\nabla K|^2+2|K|^2.\label{eq:ambient-derivative}
	\end{equation}
	Indeed, the tangent-tangent part of $DK$ is $\nabla K$, while the two terms in which one tensor slot differentiates into the radial direction are both $-K(e_a,\cdot)$, producing the factor $2|K|^2$.
	
Choose a fixed ambient orthonormal basis $E_i$ and write
$K_{ij}=K(E_i,E_j)$. Decompose each component into scalar spherical
harmonics:
\[
 K_{ij}=\sum_{\ell=0}^d(K_{ij})_\ell,\qquad
 -\Delta_{g_0}(K_{ij})_\ell=\ell(\ell+p-1)(K_{ij})_\ell.
\]
The summands are mutually orthogonal in $L^2$. Integration by parts
therefore yields
\[
\begin{aligned}
 \int_{\Sp}|\nabla K_{ij}|^2\dd\omega
 &=\sum_{\ell=0}^d\ell(\ell+p-1)
       \int_{\Sp}|(K_{ij})_\ell|^2\dd\omega\\
 &\le d(d+p-1)\sum_{\ell=0}^d
       \int_{\Sp}|(K_{ij})_\ell|^2\dd\omega\\
 &=d(d+p-1)\int_{\Sp}K_{ij}^2\dd\omega.
\end{aligned}
\]
Summing over $i,j$ gives
\[
 \sum_{i,j}\int_{\Sp}|\nabla K_{ij}|^2\dd\omega
 \le d(d+p-1)\sum_{i,j}\int_{\Sp}K_{ij}^2\dd\omega.
\]
	The left side is exactly the integral of $\sum_a|D_{e_a}K|^2$, while the right side is $d(d+p-1)|K|_{L^2}^2$. Combining \eqref{eq:y-lower-gradient}--\eqref{eq:ambient-derivative} proves \eqref{eq:spherical-EH-estimate}.
\end{proof}

\section{Exclusion of every higher normal jet}\label{sec:higher}

We are now ready to exclude a nonzero leading coefficient of degree
$d\ge2$. First we use the spherical estimate to bound the scalar term
in \eqref{eq:exact-balance}. We then obtain the opposite bound for
$\Gamma(r)$ from a weighted curvature identity on the level set.
The final comparison is strictly positive in every normal dimension.

\begin{proposition}[Higher-jet exclusion]\label{prop:higher-jet}
	If the first normal jet of $E$ vanishes along $M_-$, then $E$ has no nonzero normal Taylor coefficient of any finite order $d\ge2$.
\end{proposition}

By \eqref{eq:leading-E}, $A_{\calV}=r^dN_d+O(r^{d+1})$.
Using $h_r\to g_0$ in the angular evolution equation, we first obtain
$h_r-g_0=O(r^d)$. Substituting this back gives
\[
 \partial_rh_r=-4r^{d-1}N_d^\flat+O(r^d).
\]
Integration from $0$ to $r$ yields
\begin{equation}
	h_r=g_0+r^dK_d+O(r^{d+1}),
	\qquad
	K_d=-\frac4dN_d^\flat.\label{eq:hr-leading}
\end{equation}
By \eqref{eq:Nd-tracefree}, $K_d$ is trace-free. The round metric is a
critical point of $Y_p$, so the terms of orders $d+1,\ldots,2d-1$ in
$h_r$ contribute no linear term to $Y_p(h_r)$. Its coefficient of
$r^{2d}$ is consequently exactly the quadratic variation in $K_d$. The ambient polynomial property established after
\eqref{eq:block-Ed} applies to $K_d$. Since
\[
 \int|K_d|^2\dd\mu_0=\frac{16}{d^2}\eN,
\]
Lemma~\ref{lem:spherical-EH}, applied fiberwise and averaged over $M_-$, gives
\begin{equation}
	y_d\ge-\frac{4(d+p-1)}d\eN,\label{eq:yd-lower}
\end{equation}
where $y_d$ is the averaged coefficient of $r^{2d}$ in $Y_p(h_r)$.

\begin{lemma}[Expansion of the fiber scalar defect]\label{lem:fiber-defect-expansion}
	One has
	\begin{equation}
		\calS(r)=r^{2d}\left(\frac{2(p-1)}d\eB+y_d\right)+O(r^{2d+1}).\label{eq:S-expansion}
	\end{equation}
\end{lemma}

\begin{proof}
For $p=1$, both the intrinsic scalar curvature of the circle and
 $Y_1$ vanish, so the assertion is immediate. Suppose $p\ge2$.
 For fixed $x\in M_-$, let
	\[
	V_x(r)=\int_{\Sp}\rho(x,r\theta)\dd\omega.
	\]
	By \eqref{eq:rho-expansion},
	\[
	V_x(r)=1-a_xr^{2d}+O(r^{2d+1}),
	\qquad
	a_x=\frac1d\int_{\Sp}|E_d(x,\theta)|^2\dd\omega.
	\]
	Write $Y_p(h_r)=p(p-1)+y_xr^{2d}+O(r^{2d+1})$. Since
	\[
	\int_{\Sp}R_{h_r}\rho\dd\omega
	=V_x(r)^{(p-2)/p}Y_p(h_r),
	\]
	expand $V_x^{(p-2)/p}=1-\frac{p-2}pa_xr^{2d}+O(r^{2d+1})$ and subtract $p(p-1)V_x(r)$. The constant terms cancel, and the coefficient of $r^{2d}$ is
\[
\begin{aligned}
 y_x-\frac{p-2}{p}p(p-1)a_x+p(p-1)a_x
 &=y_x+\bigl(p-(p-2)\bigr)(p-1)a_x\\
 &=y_x+2(p-1)a_x.
\end{aligned}
\]
	Averaging over $x\in M_-$ gives \eqref{eq:S-expansion}.
\end{proof}

The density expansion also gives
\begin{equation}\label{eq:integrated-expansions}
\begin{aligned}
 \int v\rho\dd\mu_0&=r^{2d}\eB+O(r^{2d+1}),&
 \frac d{ds}\int v\rho\dd\mu_0&=2dr^{2d}\eB+O(r^{2d+1}),\\
 \int|Z|^2\rho\dd\mu_0&=r^{2d}\eZ+O(r^{2d+1}),&
 \int v^2\rho\dd\mu_0&=O(r^{4d}).
\end{aligned}
\end{equation}
Substitute these expansions and \eqref{eq:S-expansion} into
\eqref{eq:exact-balance}. After division by $r^{2d}$, the terms
$r^2\int v\rho\dd\mu_0$ and $\int v^2\rho\dd\mu_0$ tend to zero.
The remaining terms give
\[
 r^{-(2d-2)}\Gamma(r)
 =\left(4d+2(p-1)+\frac{2(p-1)}d\right)\eB+8\eZ+y_d+O(r).
\]
Using \eqref{eq:yd-lower}, we obtain
\begin{equation}
	\liminf_{r\downarrow0}r^{-(2d-2)}\Gamma(r)
	\ge\left(4d+2(p-1)+\frac{2(p-1)}d\right)\eB
	+8\eZ-\frac{4(d+p-1)}d\eN.\label{eq:Gamma-lower}
\end{equation}

We now derive the opposite estimate. Pull the horizontal metric back to $T_xM_-$ by horizontal lifting and denote it by $q_r$. The tangential part of the soliton equation and $f=r^2/4$ give
\begin{equation}
	r\partial_rq_r=-4E_{\calH}^\flat.\label{eq:horizontal-evolution}
\end{equation}
Indeed, choose horizontal lifts of fixed vector fields on $M_-$.
Their brackets with $\nu$ are vertical. The resulting correction terms
in the radial derivative of their metric pairing vanish by
orthogonality. Thus
\[
 r\partial_rq_r(X,Y)=2g(X,Y)-4\Ric(X,Y)
                   =-4g(E_{\calH}X,Y).
\]
The cancellation follows directly from
$A_{\calH}=\frac12\Id_{\calH}+E_{\calH}$.
Integrating \eqref{eq:horizontal-evolution} and using \eqref{eq:leading-E}--\eqref{eq:block-Ed} gives
\begin{equation}
	q_r=g_{M_-}+r^dQ_d+O(r^{d+1}),
	\qquad
	Q_d=-\frac4dH_d^\flat.\label{eq:qr-leading}
\end{equation}

To estimate $\Gamma(r)$, we need an integrated formula in which the
horizontal mean-curvature term cancels. We give the pointwise computation
and the weighted integration together, so that all sign and density
conventions are explicit.

\begin{lemma}[Weighted Walczak identity]\label{lem:weighted-walczak}
	For fixed $r>0$, let $T\Sigma_r=\calH\oplus\calV$ be the orthogonal splitting induced by the normal projection. Let $B_{\calH}$, $T_{\calH}$, $H_{\calH}$ be the second fundamental form, integrability tensor and mean curvature of $\calH$, and let $B_{\calV}$, $H_{\calV}$ be the corresponding objects for the integrable vertical distribution $\calV$. Then
	\begin{equation}
		H_{\calH}=\nabla_{\calV}\log\rho,\label{eq:HH-weight}
	\end{equation}
	and
	\begin{equation}
		\Gamma(r)=\int\rho\left(|B_{\calH}|^2-|T_{\calH}|^2
		+\langle\nabla_{\calH}\log\rho,H_{\calV}\rangle\right)\dd\mu_0.\label{eq:weighted-walczak}
	\end{equation}
\end{lemma}

\begin{proof}
In this proof, $\nabla$ and $\operatorname{div}$ are the connection
and divergence of the induced metric on $\Sigma_r$. For $X,Y\in\calH$,
we use the conventions
	\[
	B_{\calH}(X,Y)=\frac12\bigl((\nabla_XY)^{\calV}+(\nabla_YX)^{\calV}\bigr),
	\qquad
	T_{\calH}(X,Y)=\frac12[X,Y]^{\calV},
	\]
	and $H_{\calH}=\tr_{\calH}B_{\calH}$. The analogous tensors for $\calV$ are defined by interchanging $\calH$ and $\calV$. Since $\calV$ is tangent to the actual fibers $S_{x,r}$, it is integrable and hence $T_{\calV}=0$.
	
	The pointwise Walczak identity in these conventions is
	\begin{equation}
		\operatorname{div}_{\Sigma_r}(H_{\calH}+H_{\calV})
		=\calM+|B_{\calH}|^2+|B_{\calV}|^2-|H_{\calH}|^2-|H_{\calV}|^2-|T_{\calH}|^2.\label{eq:walczak-pointwise}
	\end{equation}
	Equivalently, using $G=|H_{\calV}|^2-|B_{\calV}|^2-\calM$,
	\begin{equation}
		G=|B_{\calH}|^2-|T_{\calH}|^2-|H_{\calH}|^2
		-\operatorname{div}_{\Sigma_r}(H_{\calH}+H_{\calV}).\label{eq:G-walczak}
	\end{equation}
Here is a component verification of \eqref{eq:walczak-pointwise}.
Use orthonormal frames $e_\alpha\in\calH$ and $e_a\in\calV$ whose
projected connection forms vanish at the point under consideration.
For this computation only, write
\[
 C_{\alpha\beta a}=\langle\nabla_{e_\alpha}e_\beta,e_a\rangle,
 \qquad D_{ab\alpha}=\langle\nabla_{e_a}e_b,e_\alpha\rangle.
\]
The symmetric and skew parts of $C$ are $B_{\calH}$ and $T_{\calH}$.
Similarly, $D$ represents $B_{\calV}$ and is symmetric because
$\calV$ is integrable. Metric compatibility gives
\[
 [e_a,e_\alpha]=-\sum_bD_{ab\alpha}e_b
                   +\sum_\beta C_{\alpha\beta a}e_\beta
\]
at that point. With the convention
$\Rm(X,Y)Z=\nabla_X\nabla_YZ-\nabla_Y\nabla_XZ-\nabla_{[X,Y]}Z$,
expansion of the mixed curvature yields
\[
\begin{aligned}
 \calM={}&\sum_{a,\alpha}
       \bigl(e_a C_{\alpha\alpha a}+e_\alpha D_{aa\alpha}\bigr)\\
 &-\sum_{a,\alpha,\beta}C_{\alpha\beta a}C_{\beta\alpha a}
  -\sum_{a,b,\alpha}D_{ab\alpha}D_{ba\alpha}.
\end{aligned}
\]
The derivative sum equals
\[
 \operatorname{div}(H_{\calH}+H_{\calV})
 +|H_{\calH}|^2+|H_{\calV}|^2.
\]
For example, the horizontal and vertical contributions to the divergence
of $H_{\calH}$ are, respectively,
\[
 \sum_\alpha\langle\nabla_{e_\alpha}H_{\calH},e_\alpha\rangle
 =-|H_{\calH}|^2,
 \qquad \sum_{a,\alpha}e_a C_{\alpha\alpha a}.
\]
Finally,
\[
 \sum C_{\alpha\beta a}C_{\beta\alpha a}
 =|B_{\calH}|^2-|T_{\calH}|^2,
 \qquad
 \sum D_{ab\alpha}D_{ba\alpha}=|B_{\calV}|^2.
\]
Substitution proves \eqref{eq:walczak-pointwise}, with the same
conventions as the integrated formula of Walczak~\cite{Walczak}.

	We now identify the weight. By the exact level-set volume law \eqref{eq:exact-volume-law}, after pulling back by $F_r(x,\theta)=\exp_x(r\theta)$,
	\begin{equation}
		F_r^*dV_{\Sigma_r}=\omega_pr^p\dd V_{M_-}\dd\omega.\label{eq:level-volume-pullback}
	\end{equation}
	On the other hand, the actual vertical metric is $r^2h_r$, so \eqref{eq:fiber-density} gives
	\begin{equation}
		dV_{\calV}=\omega_pr^p\rho\dd\omega.\label{eq:vertical-volume}
	\end{equation}
	Let $q_r$ be the horizontal metric pulled back to $T_xM_-$, and write
	\[
	dV_{q_r}=J_{\calH}\dd V_{M_-}.
	\]
	Orthogonality of $\calH$ and $\calV$ implies that the total volume density is the product of the horizontal and vertical densities. Comparing \eqref{eq:level-volume-pullback} and \eqref{eq:vertical-volume} gives the exact identity
	\begin{equation}
		J_{\calH}\rho=1.\label{eq:JH-rho}
	\end{equation}
	
	Let $U\in\calV$. Choose local vector fields $\bar E_1,\ldots,\bar E_k$ on $M_-$ and let $E_\alpha$ be their projectable horizontal lifts, orthonormal at the point under consideration. Differentiating the horizontal volume form gives
	\[
	U(\log J_{\calH})=\sum_\alpha\langle\nabla_UE_\alpha,E_\alpha\rangle.
	\]
	Because $E_\alpha$ is projectable and $U$ is vertical, $[U,E_\alpha]$ is vertical. Torsion-freeness and orthogonality therefore give
	\[
	\langle\nabla_UE_\alpha,E_\alpha\rangle
	=\langle\nabla_{E_\alpha}U,E_\alpha\rangle
	=-\langle U,\nabla_{E_\alpha}E_\alpha\rangle.
	\]
	Summing yields
	\begin{equation}
		U(\log J_{\calH})=-\langle H_{\calH},U\rangle.\label{eq:log-JH}
	\end{equation}
	Using $\log J_{\calH}=-\log\rho$ from \eqref{eq:JH-rho}, we obtain
	\[
	\langle H_{\calH},U\rangle=U(\log\rho)
	\]
	for every vertical $U$, which is \eqref{eq:HH-weight}.
	
Finally, at fixed $r$, \eqref{eq:level-volume-pullback} shows that $dV_{\Sigma_r}$ is a constant multiple of $d\mu_0$. For every tangent vector field $X$ on the compact boundaryless level
set $\Sigma_r$, the divergence theorem gives
\[
 0=\int\operatorname{div}(\rho X)\dd\mu_0
  =\int\bigl(\rho\operatorname{div}X+\langle\nabla\rho,X\rangle\bigr)
      \dd\mu_0.
\]
Using $\nabla\rho=\rho\nabla\log\rho$, we obtain
	\begin{equation}
		\int\rho\operatorname{div}X\dd\mu_0
		=-\int\rho\langle\nabla\log\rho,X\rangle\dd\mu_0.\label{eq:weighted-divergence}
	\end{equation}
	Multiply \eqref{eq:G-walczak} by $\rho$ and integrate. Taking $X=H_{\calH}+H_{\calV}$ in \eqref{eq:weighted-divergence}, and using $H_{\calH}\in\calV$, $H_{\calV}\in\calH$, together with \eqref{eq:HH-weight}, gives
	\[
	\langle\nabla\log\rho,H_{\calH}\rangle=|H_{\calH}|^2,
	\qquad
	\langle\nabla\log\rho,H_{\calV}\rangle
	=\langle\nabla_{\calH}\log\rho,H_{\calV}\rangle.
	\]
	The first term cancels the $-|H_{\calH}|^2$ in \eqref{eq:G-walczak}, leaving precisely \eqref{eq:weighted-walczak}.
\end{proof}

\begin{proof}[Proof of Proposition~\ref{prop:higher-jet}]
We next expand the terms in \eqref{eq:weighted-walczak}. In a projectable horizontal frame $X,Y$ and for $U\in\calV$, Koszul's formula gives
\[
2\langle B_{\calH}(X,Y),U\rangle=-U\bigl(q_r(X,Y)\bigr).
\]
Indeed the bracket term $\langle[X,Y],U\rangle$ is skew in $X,Y$ and therefore disappears from the symmetric part, while $[U,X]$ and $[U,Y]$ are vertical. Let $e_a$ be a round unit angular vector and choose a $g_{M_-}$-orthonormal
basis $X_\alpha$ at the base point. The corresponding unit vertical
vector is $r^{-1}(e_a+O(r^d))$, because $h_r=g_0+O(r^d)$.
Since $g_{M_-}$ is independent of $\theta$, \eqref{eq:qr-leading} gives
\[
 \left\langle B_{\calH}(X_\alpha,X_\beta),
       r^{-1}(e_a+O(r^d))\right\rangle
 =-\frac12r^{d-1}\nabla_{e_a}(Q_d)_{\alpha\beta}+O(r^d).
\]
The horizontal frame differs from a $q_r$-orthonormal frame by $O(r^d)$.
Squaring and summing over $a,\alpha,\beta$ therefore yields
\begin{equation}
	|B_{\calH}|^2=\frac14r^{2d-2}|\nabla_{\Sp}Q_d|^2+O(r^{2d-1}).\label{eq:BH-expansion}
\end{equation}
We also need to estimate the remaining weighted term. In smooth normal
bundle coordinates, a horizontal lift has bounded coefficients in the
base and angular variables. Differentiating
$\log\rho=O(r^{2d})$ horizontally therefore gives
$\nabla_{\calH}\log\rho=O(r^{2d})$.

To estimate $H_{\calV}$, fix $x\in M_-$ and consider the normal slice
$\exp_x(\nu_xM_-)$. Every radial line in this slice is an ambient geodesic.
At its center, the second fundamental form of the slice consequently
vanishes on every pair $(\xi,\xi)$; polarization shows that it is zero.
Smoothness then gives an $O(r)$ bound near the center. The radial
component of the second fundamental form of the small sphere is tangent
to the slice, so its projection to $\calH$ is zero. Thus the horizontal
mean curvature of the sphere satisfies $H_{\calV}=O(r)$, and
\[
 \langle\nabla_{\calH}\log\rho,H_{\calV}\rangle=O(r^{2d+1}).
\]
After division by $r^{2d-2}$, this term is $O(r^3)$ and tends to zero.

Discarding the nonpositive term $-|T_{\calH}|^2$ in \eqref{eq:weighted-walczak} and using \eqref{eq:BH-expansion} gives
\begin{equation}
	\limsup_{r\downarrow0}r^{-(2d-2)}\Gamma(r)
	\le\frac14\int|\nabla_{\Sp}Q_d|^2\dd\mu_0.\label{eq:Gamma-upper-one}
\end{equation}
Since the components of $Q_d$ in a fixed basis of $T_xM_-$ are spherical
polynomials of degree at most $d$, the scalar harmonic estimate gives
\[
\int|\nabla_{\Sp}Q_d|^2\dd\mu_0
\le d(d+p-1)\int|Q_d|^2\dd\mu_0.
\]
Using $Q_d=-4H_d^\flat/d$ gives
\begin{equation}
	\limsup_{r\downarrow0}r^{-(2d-2)}\Gamma(r)
	\le\frac{4(d+p-1)}d\eH.\label{eq:Gamma-upper}
\end{equation}
Combining \eqref{eq:Gamma-lower} and \eqref{eq:Gamma-upper}, and using $\eB=\eH+\eN+2\eZ$, gives
\begin{equation}
	c_{d,p}(\eH+\eN)+\left[8d+8+4(p-1)\left(1+\frac1d\right)\right]\eZ\le0,\label{eq:coefficient-inequality}
\end{equation}
where
\begin{equation}\label{eq:c-dp}
 c_{d,p}=4(d-1)+2(p-1)\left(1-\frac1d\right).
\end{equation}
Indeed, the coefficient of $\eH+\eN$ is
\[
 4d+2(p-1)+\frac{2(p-1)}d-\frac{4(d+p-1)}d=c_{d,p},
\]
and the coefficient of $\eZ$ is
$8d+8+4(p-1)(1+1/d)$. Both are strictly positive for $d\ge2$
and $p\ge1$. Since the three energies are nonnegative,
\eqref{eq:coefficient-inequality} forces $\eH=\eN=\eZ=0$,
contrary to \eqref{eq:B-energy}.
\end{proof}

\begin{remark}[The critical first mode]\label{rem:critical-mode}
	The coefficient \eqref{eq:c-dp} satisfies $c_{1,p}=0$ for every $p$. Thus the argument is exactly critical at degree one. By Lemma~\ref{lem:first-normal-jet}, this degree-one mode is precisely the second fundamental form of $M_-$. Total geodesy removes the only mode not controlled by the spherical quadratic gap.
\end{remark}

\section{Analytic continuation and global splitting}\label{sec:splitting}

The preceding argument gives vanishing to every normal order. To turn
this local information into global rigidity, we apply real analyticity
to a polynomial in the Ricci endomorphism. This avoids requiring the
auxiliary horizontal distribution to be analytic.

\begin{proof}[Proof of Theorem~\ref{thm:focal-rigidity}]
By Corollary~\ref{cor:second-order-vanishing} and Proposition~\ref{prop:higher-jet}, for every integer $j\ge1$,
\begin{equation}
	E=O(r^j)\label{eq:infinite-order}
\end{equation}
uniformly on the compact unit normal sphere bundle. All normal Taylor coefficients of $E$ vanish as smooth tensor fields along $M_-$.

We do not need to establish analyticity of $\calH$. Instead, we pass to the intrinsic polynomial in the Ricci endomorphism
\begin{equation}
	F=A^2-\frac12A.\label{eq:F-definition}
\end{equation}
Using $A=\frac12\Pi_{\calH}+E$ and $\Pi_{\calH}^2=\Pi_{\calH}$,
\begin{equation}
	F=\frac12(\Pi_{\calH}E+E\Pi_{\calH}-E)+E^2.\label{eq:F-in-E}
\end{equation}
Hence every normal Taylor coefficient of $F$ vanishes along $M_-$.

To make the passage to the full Taylor jet explicit, choose Fermi coordinates $(x,y)$ near a point of $M_-$, where $x$ are coordinates on $M_-$ and $y$ are normal coordinates. For every multi-index $\beta$ in the normal variables,
\[
\partial_y^\beta F(x,0)\equiv0
\]
as a smooth tensor-valued function of $x$. Differentiating this identity in the tangential variables shows
\[
\partial_x^\alpha\partial_y^\beta F(x,0)=0
\]
for all multi-indices $\alpha,\beta$. Thus the complete coordinate Taylor jet of $F$ vanishes at every point
of $M_-$. Vanishing of the complete jet is invariant under smooth coordinate
changes: every derivative in a new chart is a finite sum involving
components of $F$ and their derivatives in the old chart. It therefore
also holds in the analytic coordinates used below.

By Lemma~\ref{lem:proper-analytic}, $g$ is locally real analytic in
suitable coordinates. Its Ricci endomorphism and the polynomial $F$ are
therefore real analytic. A real-analytic tensor with vanishing complete Taylor jet at a point vanishes in a neighborhood of that point. By connectedness and analytic continuation,
\begin{equation}
	A^2=\frac12A
	\qquad\text{on }M.\label{eq:A-polynomial-global}
\end{equation}
Since $\tr A=k/2$, the spectrum of $A$ consists everywhere of exactly $k$ copies of $1/2$ and $m$ copies of $0$. Thus $\rank\Ric\equiv k$. The fixed-rank rigidity theorem of Fern\'andez-L\'opez and Garc\'ia-R\'io~\cite{FernandezLopezGarciaRio} gives
\[
(\widetilde M,\widetilde g)\cong(N^k,g_N)\times\R^m,
\qquad
\Ric_{g_N}=\frac12g_N.
\]
The factor $N$ is complete with positive Ricci curvature, hence compact by Bonnet--Myers. Since the universal cover and $\R^m$ are simply connected, so is $N$.
\end{proof}

\begin{proof}[Proof of Corollary~\ref{cor:cao-general}]
In dimension one the metric is flat and the soliton equation gives the
Gaussian shrinker on the universal cover, so the conclusion holds.
Assume $n\ge2$. By Lemma~\ref{lem:known-cases}, we can write
$R=k/2$ for an integer $0\le k\le n$.

If $k=0$, $k=n-1$, or $k=n$, rigidity is already known.
The case $k=1$ does not occur. Every remaining value satisfies
$2\le k\le n-2$. Setting $m=n-k$ gives $k,m\ge2$, so
Theorem~\ref{thm:focal-rigidity} applies and proves rigidity.
These cases exhaust all possible constant scalar curvatures.
\end{proof}

\begin{remark}
The proof separates the global geometry of the minimum set from the local
vanishing argument. The global normal coordinates follow from the soliton
identities and completeness. Total geodesy is then used to remove the
first normal derivative; all subsequent finite degrees are excluded by
the strict inequality \eqref{eq:coefficient-inequality}. Thus proving
total geodesy of $M_-$ under the constant-scalar-curvature assumption alone
would, together with Corollary~\ref{cor:cao-general}, establish
Cao's rigidity conjecture in full.
\end{remark}

\textbf{Declaration of AI use}. The authors used ChatGPT (OpenAI) during the preparation
of the manuscript for language editing and exploratory calculations. All mathematical
arguments and computations were independently checked by the authors.

\end{document}